\documentclass[10pt]{article}

\usepackage[english]{babel}
\usepackage{amsmath}
\usepackage{placeins}
\usepackage{amssymb}
\usepackage{latexsym}
\usepackage{amsmath}
\usepackage{amsthm}
\usepackage{enumerate}

\newtheorem{theorem}{Theorem}
\newtheorem{predl}{Theorem}
\newtheorem{lemma}{Lemma}
\newtheorem{corollary}{Corollary}

\begin{document}

\title{On the Strong Law of Large Numbers \\ for Sequences of Pairwise Independent \\ Random Variables}

\author{Valery Korchevsky\thanks{Saint-Petersburg State University of Aerospace Instrumentation, Saint-Petersburg. \endgraf E-mail: \texttt{valery.korchevsky@gmail.com} }}

\date{}

\maketitle

\begin{abstract}
We establish new sufficient conditions for the applicability of the strong law of large numbers (SLLN) for sequences of pairwise independent non-identically distributed random variables. These results generalize Etemadi's extension of Kolmogorov's SLLN for identically distributed random variables. Some of the obtained results hold with an arbitrary norming sequence in place of the classical normalization.
\end{abstract}

\bigskip

\noindent \textbf{Keywords:} strong law of large numbers, pairwise independent random variables.
{\sloppy

}

\bigskip

\bigskip

\noindent {\Large{\textbf{1. Introduction}}}

\medskip

\noindent Let $\{X_{n}\}_{n=1}^{\infty}$ be a sequence of random variables defined on the same probability space and put $S_{n}=\sum_{k=1}^{n} X_{k}$.
{\sloppy

}

The classical Kolmogorov's theorem states that if $\{X_{n}\}_{n=1}^{\infty}$ is a sequence of independent identically distributed random variables and $E|X_{1}| < \infty$ then $S_{n}/n \to EX_{1}$ almost surely. Etemadi~\cite{Etem81} generalized the Kolmogorov theorem replacing the mutual independence assumption by the pairwise independence assumption.

\begin{predl}[\cite{Etem81}]\label{P101}
Let $\{X_{n}\}_{n=1}^{\infty}$ be a sequence of pairwise independent identically distributed random variables. If $E|X_{1}| < \infty$ then $S_{n}/n \to EX_{1}$ almost surely.
{\sloppy

}

\end{predl}

One can find further extensions of the Kolmogorov theorem to wide classes of dependent random variables in the papers~\cite{Mat92} and~\cite{Mat05}.

In the present work we generalize Theorem~\ref{P101} to non-identically distributed random variables. This problem was considered in the several papers. Chandra and Goswami~\cite{ChGos92} established the following result.

\begin{predl}[\cite{ChGos92}]\label{P102}
Let $\{X_{n}\}_{n=1}^{\infty}$ be a sequence of pairwise independent random variables and put $G(x) = \sup_{n \geq 1}  P(|X_{n}| > x)$ for $x \geq 0$. If
{\sloppy

}

\begin{equation*}\label{e102}
    \int_{0}^{\infty} G(x) \, dx < \infty,
\end{equation*}

\noindent Then

\begin{equation*}\label{e103}
      \frac{1}{n} \sum_{k=1}^{n} c_{k} (X_{k} - EX_{k}) \to 0 \quad \mbox{a.s.} \quad (n \to \infty)
\end{equation*}

\noindent for each bounded sequence $\{c_{n}\}$.

\end{predl}

The following result was obtained by Bose and Chandra~\cite{BoCh94}.

\begin{predl}[\cite{BoCh94}]\label{P103}
Let $\{X_{n}\}_{n=1}^{\infty}$ be a sequence of pairwise independent random variables. Suppose that
{\sloppy

}

\begin{equation}\label{e105}
\int_{0}^{\infty} G(x) \, dx < \infty \quad \mbox{with } G(x) = \sup_{n \geq 1} \frac{1}{n} \sum_{k=1}^{n} P(|X_{k}| > x) \quad  \mbox{for all } x \geq 0,
\end{equation}

\begin{equation*}\label{e106}
\sum_{n=1}^{\infty} P(|X_{n}| > n) < \infty.
\end{equation*}

\noindent Then

\begin{equation}\label{e107}
      \frac{S_{n} - ES_{n}}{n} \to 0 \quad \mbox{a.s.}
\end{equation}

\end{predl}

Kruglov~\cite{Kruglov94} proved the next generalization of Theorem~\ref{P101}.

\begin{predl}[\cite{Kruglov94}]\label{P104}

Let $\{X_{n}\}_{n=1}^{\infty}$ be a sequence of pairwise independent random variables. Assume that

\begin{equation}\label{e108}
\sup_{n \geq 1} E|X_{n}| < \infty.
\end{equation}

\noindent If there exists a random variable $X$ such that $E|X|< \infty$ and
{\sloppy

}

\begin{equation*}\label{e109}
\sup_{n \geq 1} \frac{1}{n} \sum_{k=1}^{n} P(|X_{k}| > x) \leq C P(|X|>x) \quad \mbox{for all } x \geq 0,
\end{equation*}

\noindent where $C$ is a positive constant, then relation~\eqref{e107} holds.

\end{predl}

The aim of present work is to generalize Theorems~\ref{P103} and~\ref{P104}. We present a generalization of Theorem~\ref{P103} using an arbitrary norming sequence in place of the classical normalization. Furthermore we show that condition~\eqref{e108} in Theorems~\ref{P104} can be dropped.
{\sloppy

}

In order to prove the theorems in the present work, we use methods developed by Bose and Chandra~\cite{BoCh94} (see also Chandra~\cite{Ch12}).
{\sloppy

}

\bigskip

%\newpage

\noindent {\Large{\textbf{2. Main results}}}

\medskip

\begin{theorem}\label{T201}
Let $\{X_{n}\}_{n=1}^{\infty}$ be a sequence of pairwise independent random variables. Assume that $\{a_{n}\}_{n=1}^{\infty}$ is non-decreasing unbounded sequence of positive numbers. Suppose that
{\sloppy

}

\begin{equation}\label{e201}
\int_{0}^{\infty} G(x) \, dx < \infty \quad \mbox{with } G(x) = \sup_{n \geq 1} \frac{1}{a_{n}} \sum_{k=1}^{n} P(|X_{k}| > x) \quad  \mbox{for all } x \geq 0,
\end{equation}

\begin{equation}\label{e202}
\sum_{n=1}^{\infty} P(|X_{n}| > a_{n}) < \infty.
\end{equation}

\noindent Then

\begin{equation*}\label{e203}
      \frac{S_{n} - ES_{n}}{a_{n}} \to 0 \quad \mbox{a.s.}
\end{equation*}

\end{theorem}

Theorem~\ref{T201} generalizes Theorem~\ref{P103}, which corresponds to the case $a_{n} = n$ for all ${n \geq 1}$.
{\sloppy

}

\begin{theorem}\label{T202}
Let $\{X_{n}\}_{n=1}^{\infty}$ be a sequence of pairwise independent random variables. If there exists function $H(x)$ such that $H(x)$ is non-increasing in the interval $x \geq 0$,
{\sloppy

}

\begin{equation}\label{e205}
\int_{0}^{\infty} H(x) \, dx < \infty, \quad \mbox{and} \quad \sup_{n \geq 1} \frac{1}{n} \sum_{k=1}^{n} P(|X_{k}| > x) \leq H(x) \quad \mbox{for all } x \geq 0,
\end{equation}

\noindent then

\begin{equation*}\label{e206}
 \frac{S_{n} - ES_{n}}{n} \to 0 \quad \mbox{a.s.}
\end{equation*}

\end{theorem}

As a consequence of Theorem~\ref{T202} we immediately obtain the following result.

\begin{corollary}\label{C202}
Let $\{X_{n}\}_{n=1}^{\infty}$ be a sequence of pairwise independent random variables. If there exists a random variable $X$ such that $E|X|< \infty$ and
{\sloppy

}

\begin{equation*}\label{e207}
\sup_{n \geq 1} \frac{1}{n} \sum_{k=1}^{n} P(|X_{k}| > x) \leq C P(|X|>x) \quad \mbox{for all } x \geq 0,
\end{equation*}

\noindent where $C$ is a positive constant, then

\begin{equation*}\label{e208}
 \frac{S_{n} - ES_{n}}{n} \to 0 \quad \mbox{a.s.}
\end{equation*}

\end{corollary}

Corollary~\ref{C202} shows that we can omit condition~\eqref{e108} in Theorem~\ref{P104}.
{\sloppy

}

\bigskip

\noindent {\Large{\textbf{3. Proofs}}}

\medskip

\noindent To prove Theorems~\ref{T201} we need the following proposition that is a consequence of Theorem~1 in~\cite{ChGos92}.

\begin{lemma}\label{Lem301} Let $\{X_{n}\}_{n=1}^{\infty}$ be a sequence of non-negative random variables with finite variances. Assume that $\{a_{n}\}_{n=1}^{\infty}$ is non-decreasing unbounded sequence of positive numbers. Suppose that
{\sloppy

}

\begin{equation}\label{e302}
      Var(S_{n}) \leq C \sum_{k=1}^{n} Var(X_{k}) \qquad \mbox{for all } n \geq 1,
\end{equation}

\noindent where $C$ is a positive constant,

\begin{equation}\label{e303}
     \sum_{n=1}^{\infty} \frac{Var(X_{n})}{a^{2}_{n}} < \infty,
\end{equation}

\begin{equation*}\label{e304}
       \sup_{n \geq 1} \frac{1}{a_{n}} \sum_{k=1}^{n} EX_{n} < \infty.
\end{equation*}

\noindent Then

\begin{equation*}\label{e305}
       \frac{S_{n} - ES_{n}}{a_{n}} \to 0 \quad \mbox{a.s.}
\end{equation*}

\end{lemma}

\begin{proof}[Proof of Theorem~\ref{T201}]

Note that for pairwise independent random variables $X_{n}$, the positive parts $X_{n}^{+} := max\{0, X_{n}\}$ are pairwise independent. Likewise, the $X_{n}^{-} := max\{0, -X_{n}\}$ are pairwise independent. Thus it is enough to prove the theorem separately for the positive and negative parts. So we can assume that $X_{n} \geq 0$ for all $n \geq 1$.
{\sloppy

}

Let $Y_{n} = X_{n} \mathbb{I}_{\{X_{n} \leq a_{n}\}}$, $T_{n} = \sum_{k=1}^{n} Y_{k}$ for every $n \geq 1$. To prove the theorem , it is sufficient to show that
{\sloppy

}

\begin{equation}\label{e306}
\frac{ES_{n}-ET_{n}}{a_{n}} \to 0 \quad (n \to \infty),
\end{equation}

\begin{equation}\label{e307}
\frac{T_{n}-ET_{n}}{a_{n}} \to 0 \quad \mbox{a.s.},
\end{equation}

\begin{equation}\label{e308}
\frac{S_{n}-T_{n}}{a_{n}} \to 0 \quad \mbox{a.s.}
\end{equation}

Note that for any non-negative random variable $Z$ and $a > 0$

\begin{equation}\label{e309}
E(Z \mathbb{I}_{\{Z > a\}}) = a P(Z > a) + \int_{a}^{\infty} P(Z > x) \, dx
\end{equation}

\noindent and

\begin{equation}\label{e310}
E(Z \mathbb{I}_{\{Z \leq a\}}) \leq \int_{0}^{a} P(Z > x) \, dx.
\end{equation}

Fix an integer $N \geq 1$. Then, using~\eqref{e201} and~\eqref{e309}, for $n > N$ we obtain

\begin{multline*}
ES_{n}-ET_{n} = \sum_{k=1}^{n} E(X_{k} \mathbb{I}_{\{X_{k} > a_{k}\}}) \\ = \sum_{k=1}^{n} a_{k} P(X_{k} > a_{k}) + \sum_{k=1}^{n} \int_{a_{k}}^{\infty} P(X_{k} > x) \, dx \\ = \sum_{k=1}^{n} a_{k} P(X_{k} > a_{k}) + \sum_{k=1}^{N} \int_{a_{k}}^{\infty} P(X_{k} > x) \, dx + \sum_{k=N+1}^{n} \int_{a_{k}}^{\infty} P(X_{k} > x) \, dx \\ \leq \sum_{k=1}^{n} a_{k} P(X_{k} > a_{k}) + \sum_{k=1}^{N} \int_{0}^{\infty} P(X_{k} > x) \, dx + \sum_{k=1}^{n} \int_{a_{N}}^{\infty} P(X_{k} > x) \, dx \\ \leq \sum_{k=1}^{n} a_{k} P(X_{k} > a_{k}) + a_{N} \int_{0}^{\infty} G(x) \, dx + a_{n} \int_{a_{N}}^{\infty} G(x) \, dx.
\end{multline*}

\noindent Condition~\eqref{e202} and Kronecker's lemma (see, for example,~\cite{Petr95}) imply that

\begin{equation*}\label{e311}
\frac{1}{a_{n}} \sum_{k=1}^{n} a_{k} P(X_{k} > a_{k}) \to 0 \quad (n \to \infty).
\end{equation*}

\noindent Thus for each $N \geq 1$ we have

\begin{equation*}\label{e312}
\limsup_{n \to \infty} \frac{1}{a_{n}} (ES_{n}-ET_{n}) \leq \int_{a_{N}}^{\infty} G(x) \, dx
\end{equation*}

\noindent so we get assertion~\eqref{e306} by letting $N \to \infty$.

To establish~\eqref{e307} we shall prove that conditions of Lemma~\ref{Lem301} are satisfied for sequence $\{Y_{n}\}_{n=1}^{\infty}$. It follows from pairwise independence of random variables $Y_{n}$ that assertion~\eqref{e302} is satisfied for sequence $\{Y_{n}\}_{n=1}^{\infty}$.
{\sloppy

}

Using~\eqref{e201} and~\eqref{e310}, we obtain

\begin{multline*}
\sum_{n=1}^{\infty} \frac{Var(Y_{n})}{a_{n}^{2}} \leq \sum_{n=1}^{\infty} \frac{E \left(X_{n}^{2} \mathbb{I}_{\{X_{n} \leq a_{n}\}} \right)}{a_{n}^{2}} \leq \sum_{n=1}^{\infty} \frac{1}{a_{n}^{2}} \int_{0}^{a_{n}^{2}} P(X_{n} > x^{1/2}) \, dx \\ \leq 2 \sum_{n=1}^{\infty} \frac{1}{a_{n}^{2}} \int_{0}^{a_{n}} y P(X_{n} > y) \, dy \leq 4 \sum_{n=1}^{\infty} \frac{1}{([a_{n}]+1)^{2}} \int_{0}^{[a_{n}]+1} y P(X_{n} > y) \, dy \\ \leq 8 \sum_{n=1}^{\infty} \sum_{j=[a_{n}]+1}^{\infty} \frac{1}{j^{3}} \sum_{k=1}^{[a_{n}]+1} \int_{k-1}^{k} y P(X_{n} > y) \, dy \\ \leq 8 \sum_{n=1}^{\infty} \sum_{j=[a_{n}]+1}^{\infty} \frac{1}{j^{3}} \sum_{k=1}^{j} \int_{k-1}^{k} y P(X_{n} > y) \, dy \\ \leq 8 \sum_{j=[a_{1}]+1}^{\infty} \sum_{n:a_{n} \leq j} \frac{1}{j^{3}} \sum_{k=1}^{j} \int_{k-1}^{k} y P(X_{n} > y) \, dy \\ \leq 8 \sum_{j=1}^{\infty} \sum_{k=1}^{j} \frac{1}{j^{2}} \int_{k-1}^{k} y \frac{\sum_{n:a_{n} \leq j} P(X_{n} > y)}{j} \, dy \\ \leq 8 \sum_{j=1}^{\infty} \sum_{k=1}^{j} \frac{1}{j^{2}} \int_{k-1}^{k} y G(y) \, dy = 8 \sum_{k=1}^{\infty} \sum_{j=k}^{\infty} \frac{1}{j^{2}} \int_{k-1}^{k} y G(y) \, dy \\ \leq 16 \sum_{k=1}^{\infty} \frac{1}{k} \int_{k-1}^{k} y G(y) \, dy \leq 16 \sum_{k=1}^{\infty} \int_{k-1}^{k} G(y) \, dy = 16 \int_{0}^{\infty} G(y) \, dy < \infty.
\end{multline*}

\noindent where $[a_{n}]$ is the integer part of $a_{n}$. Hence condition~\eqref{e303} is satisfied for sequence $\{Y_{n}\}_{n=1}^{\infty}$.
{\sloppy

}

For each $n \geq 1$ we have

\begin{equation*}\label{e314}
\frac{1}{a_{n}} \sum_{k=1}^{n} EY_{k} \leq \frac{1}{a_{n}} \sum_{k=1}^{n} EX_{k} = \frac{1}{a_{n}} \sum_{k=1}^{n} \int_{0}^{\infty} P(X_{k} > x) \, dx \leq \int_{0}^{\infty} G(x) \, dx < \infty.
\end{equation*}

\noindent Thus the sequence of random variables $\{Y_{n}\}_{n=1}^{\infty}$ satisfies conditions of Lemma~\ref{Lem301}, so relation~\eqref{e307} holds.
{\sloppy

}

To complete the proof it remains to verify assertion~\eqref{e308}. Using~\eqref{e202}, we obtain

\begin{equation*}\label{e315}
\sum_{n=1}^{\infty} P(X_{n} \ne Y_{n}) = \sum_{n=1}^{\infty} P(X_{n} > a_{n}) < \infty.
\end{equation*}

\noindent From Borel--Cantelli lemma and relation~\eqref{e307} it follows that

\begin{equation}\label{e316}
\frac{S_{n}-ET_{n}}{a_{n}} \to 0 \quad \mbox{a.s.}
\end{equation}

\noindent Thus~\eqref{e308} follows from~\eqref{e307} and~\eqref{e316}.
\end{proof}

\begin{proof}[Proof of Theorem~\ref{T202}]

As $H(x)$ is non-increasing in the interval $x \geq 0$, condition $\int_{0}^{\infty} H(x) \, dx < \infty$ implies that $\sum_{k=0}^{\infty} 2^{k} H(2^{k}) < \infty$. Thus using~\eqref{e205} we have

\begin{multline*}\label{e317}
\sum_{n=2}^{\infty} P(|X_{n}|>n) = \sum_{k=0}^{\infty} \sum_{n=2^{k}+1}^{2^{k+1}} P(|X_{n}|>n) \\ \leq \sum_{k=0}^{\infty} \sum_{n=2^{k}+1}^{2^{k+1}} P(|X_{n}|>2^{k}) \leq \sum_{k=0}^{\infty} \sum_{n=1}^{2^{k+1}} P(|X_{n}|>2^{k}) \leq \sum_{k=0}^{\infty} 2^{k+1} H(2^{k}) < \infty.
\end{multline*}

\noindent Now, taking into account that~\eqref{e205} implies~\eqref{e105}, we can conclude that the desired result follows from Theorem~\ref{P103}.
\end{proof}

\end{document}